\documentclass[11pt]{amsart}
\usepackage{lmodern}
\usepackage{amsmath, amsthm, amssymb, amsfonts}
\usepackage[normalem]{ulem}
\usepackage{hyperref}

\usepackage{mathrsfs}

\usepackage{verbatim} 
\usepackage{longtable}

\usepackage{mathtools}

\usepackage{tikz}
\usetikzlibrary{decorations.pathmorphing}
\tikzset{snake it/.style={decorate, decoration=snake}}

\usepackage{caption}

\usepackage{tikz-cd}
\usetikzlibrary{arrows}

\theoremstyle{plain}
\newtheorem{thm}{Theorem}[section]
\newtheorem{cor}[thm]{Corollary}
\newtheorem{lem}[thm]{Lemma}
\newtheorem{prop}[thm]{Proposition}

\theoremstyle{definition}

\theoremstyle{remark}
\newtheorem{rmk}[thm]{Remark}

\newcommand{\BA}{{\mathbb{A}}}

\newcommand{\BC}{{\mathbb{C}}}

\newcommand{\BF}{{\mathbb{F}}}
\newcommand{\BG}{{\mathbb{G}}}

\newcommand{\BK}{{\mathbb{K}}}

\newcommand{\BN}{{\mathbb{N}}}

\newcommand{\BP}{{\mathbb{P}}}
\newcommand{\BQ}{{\mathbb{Q}}}

\newcommand{\BT}{{\mathbb{T}}}

\newcommand{\BZ}{{\mathbb{Z}}}

\newcommand{\CC}{{\mathcal C}}
\newcommand{\CD}{{\mathcal D}}
\newcommand{\CE}{{\mathcal E}}
\newcommand{\CF}{{\mathcal F}}
\newcommand{\CG}{{\mathcal G}}

\newcommand{\CL}{{\mathcal L}}

\newcommand{\CO}{{\mathcal O}}

\newcommand{\CT}{{\mathcal T}}
\newcommand{\CU}{{\mathcal U}}
\newcommand{\CV}{{\mathcal V}}

\DeclareFontFamily{OT1}{rsfs}{}
\DeclareFontShape{OT1}{rsfs}{n}{it}{<-> rsfs10}{}
\DeclareMathAlphabet{\curly}{OT1}{rsfs}{n}{it}

\newcommand{\git}{\mathbin{
  \mathchoice{/\mkern-6mu/}
    {/\mkern-6mu/}
    {/\mkern-5mu/}
    {/\mkern-5mu/}}}

\usepackage{tikz}
\usepackage{lmodern}
\usetikzlibrary{decorations.pathmorphing}

\begin{document}
\title[Cohomology of the moduli of Higgs bundles ]{Cohomology of the moduli of Higgs bundles \\ on a curve via positive characteristic}
\date{\today}

\author[M. A. de Cataldo]{Mark Andrea de Cataldo}
\address{Stony Brook University}
\email{mark.decataldo@stonybrook.edu}

\author[D. Maulik]{Davesh Maulik}
\address{Massachusetts Institute of Technology}
\email{maulik@mit.edu}

\author[J. Shen]{Junliang Shen}
\address{Massachusetts Institute of Technology}
\email{jlshen@mit.edu}
\address{Yale University}
\email{junliang.shen@yale.edu}

\author[S. Zhang]{Siqing Zhang}
\address{Stony Brook University}
\email{siqing.zhang@stonybrook.edu}

\begin{abstract}
For a curve of genus $g$ and any two degrees coprime to the rank, we construct a family of ring isomorphisms parameterized by the complex Lie group $\mathrm{GSp}(2g),$ between the cohomology of the moduli spaces of stable Higgs bundles which preserve the perverse filtrations. As a consequence, we prove two structural results concerning the cohomology of Higgs moduli which are predicted by the P=W Conjecture in Non-Abelian Hodge Theory: (1) Galois conjugation for character varieties preserves the perverse filtrations for the corresponding Higgs moduli spaces. (2) The restriction of the Hodge--Tate decomposition for a character variety to each piece of the perverse filtration for the corresponding Higgs moduli space also gives a decomposition.

Our proof uses reduction to positive characteristic and relies on the non-abelian Hodge correspondence in characteristic $p$ between Dolbeault and de Rham moduli spaces.
\end{abstract}

\maketitle

\setcounter{tocdepth}{1} 

\tableofcontents
\setcounter{section}{-1}

\section{Introduction}


\subsection{Cohomology of the moduli of Higgs bundles} 
Let $C$ be a nonsingular irreducible projective curve over the complex numbers $\BC$ of genus $g \geq 2$. Throughout the paper, we fix a positive integer $n >0$. For any positive integer $d$ coprime to $n$, we denote by $M_{\mathrm{Dol}}(C,d)$ the moduli of (slope-)stable Higgs bundles of rank $n$ and degree $d$:
\[
(\CE, \theta):  \CE \xrightarrow{\theta} \CE \otimes \Omega_C, \quad \quad \mathrm{rank}(\CE) = n, \quad \mathrm{deg}(\CE) = d.
\]
It is a nonsingular quasi-projective variety which admits a Lagrangian fibration
\[
h: M_{\mathrm{Dol}}(C,d) \to A: = \oplus_{i=1}^n H^0(C, \Omega^{\otimes i}_C), \quad (\CE,\theta) \mapsto \mathrm{char}(\theta) \in A
\]
known as the Hitchin fibration \cite{Hit,Hit1}. Here $\mathrm{char}(\theta)$ stands for the coefficients of the characteristic polynomial of the Higgs field $\theta$, \emph{i.e.,}
\[
\mathrm{char}(\theta) = (a_1, a_2, \dots, a_n), \quad a_i: = \mathrm{trace}(\wedge^i \theta) \in H^0(C, \Omega^{\otimes i}_C).
\]
The singular cohomology of $M_{\mathrm{Dol}}(C,d)$ carries an increasing filtration 
\[
P_0H^*(M_{\mathrm{Dol}}(C,d), \BC) \subset P_1H^*(M_{\mathrm{Dol}}(C,d), \BC)  \subset  \cdots \subset H^*(M_{\mathrm{Dol}}(C,d), \BC)
\]
called the \emph{perverse filtration} which is governed by the topology of the Hitchin fibration \cite{dCHM1}.

The cohomology, along with the perverse filtration, has been the subject of intense study in recent years, motivated by the so-called P=W Conjecture in Non-Abelian Hodge Theory, which we recall in Section \ref{Sec0.3}.

\subsection{Main results}\label{s0.4}

Before stating our main results, we first recall a natural set of generators for the cohomology algebra of $M_{\mathrm{Dol}}(C,d)$.

Let $(\CU, \theta)$ be a universal family over $C \times M_{\mathrm{Dol}}(C,d)$. A natural way to construct cohomology classes on $M_{\mathrm{Dol}}(C,d)$ is to integrate the $k$-th component $\mathrm{ch}_k(\CU)$ over a class $\gamma$ on $C$, so that the resulting class is determined by $k \in \BN$ and $\gamma  \in H^*(C, \BC)$. While the choice of a universal family is not unique, we may normalize the universal family to get tautological classes
\begin{equation}\label{taut}
c(\gamma, k) \in H^*(M_{\mathrm{Dol}}(C,d), \BC), \quad \gamma \in H^*(C, \BC),\quad k \in \BN
\end{equation}
so that they do not depend on the choice of a universal family; see \cite[Section 0.3]{dCMS} and Section \ref{1.1}. 
Markman \cite{Markman} proved that the tautological classes (\ref{taut}) generate $H^*(M_{\mathrm{Dol}}(C,d), \BC)$ as a $\BC$-algebra.

Let $\Lambda: = H^1(C, \BZ)$ be the lattice with the intersection pairing, and let $\Lambda_\BK$ be $\Lambda \otimes \BK$ for $\BK$ a field.
We consider the similitude group 
\begin{equation}\label{gsp}
\mathrm{GSp}(\Lambda_\BC) = \Big{\{}A\in \mathrm{GL}_{2g}(\Lambda_\BC)|~~~ \exists \lambda_A\in \BC^{\ast}, \langle Aw_1, Aw_2\rangle = \lambda_A \langle w_1, w_2\rangle, ~~\forall w_1, w_2 \in \Lambda_\BC \Big{\}}
\end{equation}
which is a subgroup of $\mathrm{GL}_{2g}(\Lambda_\BC)$. Here,
$\langle -,-\rangle$ is the intersection pairing on the vector space $\Lambda_\BC$, and $\lambda_A$ is a non-zero constant uniquely determined by $A$. The symplectic group $\mathrm{Sp}(\Lambda_\BC) \subset \mathrm{GSp}(\Lambda_\BC)$ is characterized by $\lambda_A = 1$. The natural action of $\mathrm{GSp}(\Lambda_\BC)$ on $\Lambda_\BC$ admits an extension to the total cohomology of the curve 
\[
H^*(C, \BC) = H^0(C, \BC) \oplus \Lambda_\BC \oplus H^2(C, \BC)
\]
which acts as $\mathrm{id}$ on $H^2(C, \BC)$ and as multiplication by $\lambda_A$ on $H^0(C, \BC)$.  We note that this choice of extension of the $\mathrm{GSp}(\Lambda_\BC)$-action to $H^*(C, \BC)$ is not the natural extension compatible with cup product; nevertheless, this choice is essential for our purpose in view of Theorem \ref{thm3.1}; see Remark \ref{rmk3.1}. For $A\in \mathrm{GSp}(\Lambda_\BC)$ and $\gamma \in H^*(C, \BC)$, we use $A\gamma \in H^*(C, \BC)$ to denote the class given by the action of $A$ on $\gamma$. 

The main result of this paper is the following:

\begin{thm}\label{main_thm}

\begin{enumerate}
\item[(a)]
For each integer $d$ coprime to $n$, there exists an action of $\mathrm{GSp}(\Lambda_\BC)$ 
on $H^*(M_\mathrm{Dol}(C,d), \BC)$ by $\BC$-algebra automorphisms.  For each $A \in \mathrm{GSp}(\Lambda_\BC)$, the corresponding automorphism $G_A$ acts on tautological generators compatibly with the action on $H^*(C, \BC)$:
\[
G_A(c(\gamma, k )) = c(A\gamma, k).
\]
Furthermore, it preserves the perverse filtrations:
\[
G_A\left(P_kH^i(M_\mathrm{Dol}(C,d), \BC)\right) = P_kH^i(M_\mathrm{Dol}(C,d'), \BC), \quad \forall k,i.
\]
\item[(b)]
Given two integers $d$ and $d'$ coprime to $n$, there exists a $\BC$-algebra isomorphism
\[
\phi_{d,d'}: H^*(M_\mathrm{Dol}(C,d), \BC) \xrightarrow{\simeq} H^*(M_\mathrm{Dol}(C,d'), \BC)
\]
which acts on tautological generators via the relation
\[
\phi_{d,d'}(c(\gamma, k )) = c(\gamma, k).
\]
Furthermore, this isomorphism intertwines the actions of $\mathrm{GSp}(\Lambda_\BC)$ constructed in part (a) for $d$ and $d',$ and also preserves the perverse filtrations.
\end{enumerate}
\end{thm}

We note that the action described above is uniquely characterized by its action on the tautological generators. However, it is not obvious that such an action exists, \emph{i.e.}, that it is compatible with the relations between these tautological generators or preserves the perverse filtration.

In part (a), if we restrict to the arithmetic subgroup $\mathrm{Sp}(\Lambda) \subset \mathrm{GSp(\Lambda_\BC)}$, the corresponding action can be constructed geometrically as we describe next.  If we take the universal curve
$\mathcal{C} \rightarrow \mathcal{M}_g$, then the monodromy of the corresponding family of Higgs moduli spaces defines an action of
the mapping class group $\pi_1(\mathcal{M}_g)$ on $H^*(M_\mathrm{Dol}(C,d), \BC)$.  This action is compatible with perverse filtrations by the main result of \cite{dCM}.  Finally, since the kernel of the symplectic representation of the mapping class group
has no effect on  the tautological generators,  this action descends to an action of $\mathrm{Sp}(\Lambda)$.

We do not know a direct construction of the assignment $A \mapsto G_A$ for $A \in \mathrm{GSp}(\Lambda_\BC).$  Instead, our approach is to use techniques from positive characteristic to construct some of these $G_A$, and to then couple this construction with  density arguments.

\subsection{Implications for the P=W Conjecture}\label{Sec0.3}

The motivation for our results in this paper comes from the P=W Conjecture in Non-Abelian Hodge Theory, proposed in 2010 by  de Cataldo, Hausel, and Migliorini \cite{dCHM1}.  This conjecture predicts that the topology of the Hitchin fibration (particularly the perverse filtration) interacts surprisingly, via \emph{Non-Abelian Hodge Theory}, with the Hodge theory of the corresponding character variety (particularly the weight filtration).

More precisely, we consider the character variety $M_{\mathrm{B}}(C,d)$ of rank $n$ and degree $d$:
\begin{equation*}
M_{\mathrm{B}}(C,d) := \Big{\{}a_k, b_k \in \mathrm{GL}_n,~k=1,2,\dots,g: ~~\prod_{j=1}^g [a_j, b_j] = \zeta_n^d \cdot \mathrm{Id}_n \Big{\}}\git \mathrm{GL}_n,\quad \zeta_n:=e^{\frac{2\pi \sqrt{-1} }{n}},
\end{equation*}
obtained as an affine $\mathrm{GIT}$ quotient with respect to the action by conjugation. Non-Abelian Hodge Theory \cite{Si1994II} (see \cite{HT1} for the twisted case) then induces a diffeomorphism between the moduli spaces $M_\mathrm{Dol}(C,d)$ and $M_{\mathrm{B}}(C,d)$, which identifies their cohomology rings:
\begin{equation}\label{non-ab}
    H^*(M_\mathrm{Dol}(C,d), \BC) = H^*(M_{\mathrm{B}}(C,d), \BC).
\end{equation}
The P=W Conjecture refines (\ref{non-ab}) incorporating the perverse filtration for $M_\mathrm{Dol}(C,d)$ and the mixed Hodge structure for $M_{\mathrm{B}}(C,d)$; it predicts that
\begin{equation}\label{P=W}
P_k H^*(M_\mathrm{Dol}(C,d), \BC) = W_{2k} H^*(M_{\mathrm{B}}(C,d), \BC) =  W_{2k+1} H^*(M_{\mathrm{B}}(C,d), \BC), \quad \forall k
\end{equation}
with $W_\bullet$ the weight filtration.

The P=W Conjecture makes the following predictions for the perverse filtration, which we will deduce as consequences of our main theorem.

\subsubsection{Galois}
The first prediction comes from Galois conjugation on the character variety. 

Fix $n$, and consider two integers $d,d'$ coprime to $n$. The Betti moduli spaces $M_{\mathrm{B}}(C,d)$ and $M_{\mathrm{B}}(C,d')$ are Galois conjugate via an automorphism of $\BQ[\zeta]$ sending $\zeta^d$ to $\zeta^{d'}$ \cite[Section 4]{Survey}. Galois conjugation induces an isomorphism preserving the weight filtrations
\begin{equation}\label{Galoi0}
\widetilde{\phi}_{d,d'}:H^*(M_{\mathrm{B}}(C,d) , \BC) \xrightarrow{\simeq} H^*(M_{\mathrm{B}}(C,d'),\BC),\quad W_k \mapsto W_k.
\end{equation}
By passing through the non-abelian Hodge isomorphism (\ref{non-ab}), this induces a ring isomorphism between the cohomology groups of the Dolbeault moduli spaces
\begin{equation}\label{Galoi1}
\widetilde{\phi}_{d,d'}: H^*(M_\mathrm{Dol}(C,d), \BC) \xrightarrow{\simeq} H^*(M_\mathrm{Dol}(C,d'), \BC),
\end{equation}
which, assuming the P=W Conjecture, should preserve the respective perverse filtrations.  

The first consequence of our main theorem is that this is indeed the case:
\begin{thm}\label{thm0.2}
Galois conjugation $\widetilde{\phi}_{d,d'}$  (\ref{Galoi1}) preserves perverse filtrations:
\begin{equation}\label{galois}
    \widetilde{\phi}_{d,d'}\left(
    P_kH^i(M_\mathrm{Dol}(C,d), \BC)\right) = P_kH^i(M_\mathrm{Dol}(C,d'), \BC), \quad \forall k,i.
\end{equation}
\end{thm}

We will show this in Sections \ref{sec1.2} and \ref{5.2} by matching $\widetilde{\phi}_{d,d'}$ (\ref{Galoi1}) and $\phi_{d,d'}$ of Theorem \ref{main_thm} (b). The following corollary of Theorem \ref{thm0.2} is immediate. 

\begin{cor}[]
For fixed rank $n$, if the P=W Conjecture holds for some degree $d$ coprime to $n$, then it holds for every degree $d'$ coprime to $n$.
\end{cor}

\subsubsection{Weight decomposition}

The second prediction of Theorem \ref{main_thm} relates the perverse filtration with the weight decomposition of the character variety.

The Hodge structure for $M_{\mathrm{B}}(C,d)$ was shown to be of Hodge--Tate type \cite{Shende}, and we have a canonical
Hodge--Tate decomposition
\begin{equation}\label{Hodge-Tate}
H^*(M_{\mathrm{B}}(C,d), \BC) = \bigoplus_{i,k} \mathrm{Hdg}^i_k, \quad \mathrm{Hdg}^i_k:= W_{2k}\cap F^k \left(H^i(M_{\mathrm{B}}(C,d), \BC)\right).
\end{equation}
In particular, we have a canonical decomposition $W_{2k}H^i = \oplus_{k' \leq k} \mathrm{Hdg}^i_{k'}$ for the weight filtration.  
The P=W Conjecture implies immediately the same equality with $P_k$ replacing $W_{2k}.$

Using our main Theorem \ref{main_thm}, we show the following weaker  compatibility of the perverse filtration on the Dolbeault side, with the
Hodge--Tate decomposition (\ref{Hodge-Tate}) on the Betti side.

\begin{thm}\label{thm0.4}
 We have the following compatibility between 
 the decomposition (\ref{Hodge-Tate})
 and the perverse filtration:
 \[
P_kH^i(M_{\mathrm{Dol}}(C,d), \BC) = \bigoplus_{k' } \left(P_kH^i(M_{\mathrm{Dol}}(C,d), \BC) \right) \cap \mathrm{Hdg}^{i}_{k'}.
\]
\end{thm}
Here we use the non-abelian Hodge isomorphism (\ref{non-ab}) to transfer the decomposition (\ref{Hodge-Tate}) to $H^*(M_{\mathrm{Dol}}(C,d), \BC).$
In general, the restriction of a decomposition of a vector space $Q = \oplus_i Q_i$ to a sub-vector space $Q' \subset Q$ may fail to be a decomposition, \emph{i.e.}, in general $Q'\neq \oplus_i (Q' \cap Q_i)$.
We prove that instead this is the case for the subspaces of the perverse filtration in Section \ref{5.2}, by relating the $\BG_m$ which induces the Hodge-Tate decomposition with the central $\BG_m$ in the $\mathrm{GSp}$-action we construct.

Finally, the perverse filtration admits a natural splitting, known as the first Deligne splitting. In Section \ref{5.3}, we explain  that the operators of our main theorem are compatible with this splitting; see Corollary \ref{firstdsp}.

\subsection{Acknowledgements}
We thank the anonymous referee for careful
reading and numerous useful suggestions. M.A. de Cataldo is  partially supported  by a Simons Fellowship in Mathematics, and by  NSF grant DMS-1901975; he thanks the I.A.S. in  Princeton for the perfect working conditions during a revision of this paper.
J. Shen is supported by  NSF grant DMS-2000726 and  NSF grant DMS-2134315. S. Zhang is partially supported by NSF grant DMS-1901975.

\section{Tautological classes}

\subsection{Tautological classes}\label{1.1}
Let $C$ be a complex curve as in the introduction; we assume that $(n,d) = 1$. We first review the construction of the tautological classes
\[
c(\gamma, k) \in H^*(M_{\mathrm{Dol}}(C,d), \BC), \quad \gamma \in H^*(C, \BC), ~~~ k \in \BZ_{\geq 0}
\]
as integrals of normalized classes, following \cite{dCMS}.
In this paper, as it is standard, when we push-forward cohomology classes,
it is always for proper l.c.i. morphisms, and
we use the proper push-forward  in cohomology stemming from  Chow bi-variant theory coupled with the cycle class map,
as in \cite{Fulton}.

Let
\[
p_C: C \times M_{\mathrm{Dol}}(C,d) \to C,\quad  p_M: C \times M_{\mathrm{Dol}}(C,d) \to  M_{\mathrm{Dol}}(C,d)\]
be the projections. We say that a triple $(\CU, \theta, \alpha)$ is a {\em twisted universal family} over $C \times M_{\mathrm{Dol}}(C,d)$, if $(\CU, \theta)$ is a universal family and 
\begin{equation}\label{form0}
\alpha = p_C^\ast \alpha_C +p_M^\ast \alpha_{M} \in H^2(C\times M_{\mathrm{Dol}}(C,d), \BC)
\end{equation}
with $\alpha_C \in H^2(C, \BC)$ and $\alpha_M \in H^2(M_{\mathrm{Dol}}(C,d), \BC)$. For a twisted universal family $(\CU, \theta, \alpha)$, we define the {\em twisted Chern character} $\mathrm{ch}^\alpha(\CU)$ as
\begin{equation*}\label{uni_eqn1}
\mathrm{ch}^\alpha(\CU) = \mathrm{ch}(\CU) \cup \mathrm{exp}(\alpha) \in H^*(C\times M_{\mathrm{Dol}}(C,d), \BC),
\end{equation*}
and we denote by 
\[
\mathrm{ch}_k^\alpha(\CU) \in H^{2k}(C\times M_{\mathrm{Dol}}(C,d),\BC)
\]
its degree $2k$ part. The class $\mathrm{ch}^\alpha(\CU)$ is called \emph{normalized} if \begin{equation}\label{Condition0}
\mathrm{ch}^\alpha_1(\CU)|_{x \times M_{\mathrm{Dol}}(C,d)} =0 \in H^2(M_{\mathrm{Dol}}(C,d), \BC), \qquad \mathrm{ch}^\alpha_1(\CU)|_{C \times y} =0 \in H^2(C, \BC),
\end{equation}
with $x \in C$ and $y\in M_{\mathrm{Dol}}(C,d)$ points. For two universal families $(\CU_1, \theta_1)$ and $(\CU_2, \theta_2)$, there is a line bundle $\CL$ pulled back from $M_{\mathrm{Dol}}(C,d)$ so that $\CU_1 = \CU_2 \otimes \CL$. By the condition (\ref{Condition0}), normalized classes do not depend on the choice of a universal family. For any $\gamma \in H^i(C, \BQ)$, the \emph{tautological class} $c(\gamma ,k)$ is defined by integrating the degree $k$ normalized class 
\begin{equation}\label{taut_class}
c(\gamma ,k) := \int_\gamma \mathrm{ch}_k^\alpha(\CU) = {p_{M}}_\ast( p_C^\ast \gamma \cup \mathrm{ch}^\alpha_k(\CU)) \in H^{i+2k-2}(M_{\mathrm{Dol}}(C,d), \BC).
\end{equation}

There is an alternative way in \cite{HT1} to obtain canonically defined classes in $H^*(M_{\mathrm{Dol}}(C,d), \BC)$ which we briefly review; this will only be used in Section \ref{sec1.2} to characterize the action of Galois conjugation on the tautological classes of the character varieties. We let $\CT$ be the projective bundle $\BP(\CU)$ associated with \emph{any} universal family $(\CU, \theta)$. If we assume that $\xi_1, \dots, \xi_n$ are the Chern roots of $\CU$, then the Chern roots for $\CT$ are 
\[
\xi_1 - \overline{\xi},~~ \xi_2 - \overline{\xi},~~ \dots,~~ \xi_n - \overline{\xi},
\]
with $\overline{\xi}$ the average of the $\xi_i$. We may consider Chern classes $c_k(\CT)$ and Chern characters $\mathrm{ch}_k(\CT)$ via the Chern roots. In particular $c_1(\CT) = 0$. For any twisted universal family $(\CU, \theta, \alpha)$, we have
\begin{equation}\label{eqn10}
\mathrm{ch}^\alpha (\CU) = \mathrm{ch}(\CT)\cup \mathrm{exp}\left( \frac{c_1(\CU)}{n} + \alpha \right),
\end{equation}
and for a normalized class $\mathrm{ch}^\alpha(\CU)$ we have
\[
\mathrm{ch}^\alpha_1(\CU) = c_1(\CU) + n \alpha \in H^1(C, \BC) \otimes H^1(M_{\mathrm{Dol}}(C,d), \BC).
\]
Therefore the degree 1 tautological classes 
\begin{equation}\label{H^1}
c(\gamma, 1) \in H^1(M_{\mathrm{Dol}}(C,d), \BC), \quad \gamma \in H^1(C, \BC)
\end{equation}
recover all the classes in $H^1(M_{\mathrm{Dol}}(C,d), \BC)$, and by (\ref{eqn10}) any tautological class (\ref{taut_class}) can be expressed in terms of (\ref{H^1}) and
\begin{equation}\label{taut_T}
\int_\gamma c_k(\CT) \in H^*(M_{\mathrm{Dol}}(C,d), \BC), \quad \quad \gamma \in H^*(C, \BC),~~~ k\geq 2.
\end{equation}

\begin{rmk}
Recall that we have the product formula (see the equation following Remark 2.4.4 in \cite{dCHM1}):
\[
H^*(M_{\mathrm{Dol}}(C,d), \BC) = H^*(\mathrm{Jac}(C), \BC) \otimes H^*(\hat{M}_{\mathrm{Dol}}(C,d), \BC)
\]
with $\hat{M}_{\mathrm{Dol}}(C,d)$ the $\mathrm{PGL}_n$-Higgs moduli space and $\mathrm{Jac}(C)$ the Jacobian of the curve $C$. The classes (\ref{H^1}) generate the first factor; in turn this is generated by the tautological classes associated with the normalized Poincar\'e line bundle over $C\times \mathrm{Jac}(C)$. The classes (\ref{taut_T}) generate the second factor.
\end{rmk}

\subsection{Character varieties}\label{sec1.2}
In \cite{HT1}, Hausel and Thaddeus described the tautological classes directly on the character variety $M_{\mathrm{B}}(C,d)$ side. By their description, these classes are preserved under Galois conjugation (\ref{Galoi0}).


\begin{prop}\label{Prop1.2} For $d$ and $d'$ coprime to $n$, a morphism of $\BC$-algebras
\[
H^*(M_\mathrm{Dol}(C,d), \BC) \xrightarrow{\simeq} H^*(M_\mathrm{Dol}(C,d'), \BC)
\]
is induced by Galois conjugation (\ref{Galoi0}) if and only if it preserves the tautological classes (\ref{taut_class})
\[
c(\gamma, k)  \in H^*(M_{\mathrm{Dol}}(C,d), \BC) \mapsto c(\gamma, k)  \in H^*(M_{\mathrm{Dol}}(C,d'), \BC).
\]
\end{prop}

This proposition was observed by Hausel \cite[Remark 4.8]{Survey}; we give a proof here for the reader's convenience.

\begin{proof}
We first review the construction of Hausel--Thaddeus \cite{HT1}. We consider the map
\[
\mu: \mathrm{GL}_n^{2g} \to \mathrm{GL}_n, \quad (a_1, \dots, a_g, b_1, \dots, b_g) \mapsto \prod_{j=1}^g [a_j, b_j],
\]
so that $M_{\mathrm{B}}(C,d)$ is the geometric quotient of $\mu^{-1}(\zeta_n^d\mathrm{Id}_n)$ by the conjugation action. Now we follow \cite[Section 1]{HT1} to describe the principal $\mathrm{PGL}_n$-bundle $\CT'_d$  corresponding to $\CT$ in Section \ref{1.1}.

Let $\widetilde{C} \to C$ be the universal cover with the natural $\pi_1(C)$-action on $\widetilde{C}$. Then there is a action of the group $\pi_1(C) \times \mathrm{GL}_n$ on the product
\[
\mathrm{PGL}_n \times \mu^{-1}(\zeta_n^d\mathrm{Id}_n) \times \widetilde{C},
\]
given by
\[
(p,g) \cdot(h, \rho, x) = \left([g]\rho(p)h, [g]\rho [g]^{-1}, p\cdot x\right).
\]
Here $[g]$ is the projection of $g\in \mathrm{GL}_n$ to $\mathrm{PGL}_n$, and we view $\rho \in \mu^{-1}(\zeta_n^d\mathrm{Id}_n)$  as a homomorphism $\pi_1(C) \to \mathrm{PGL}_n$. The resulting quotient, denoted by $\CT'_d$, gives the desired $\mathrm{PGL}_n$-principal bundle over the product $M_{\mathrm{B}}(C,d) \times C$. It corresponds to $\CT = \BP(\CU)$ on the Higgs side via the diffeomorphism given by Non-Abelian Hodge Theory; we refer to \cite[Section 5]{HT1} for more details.

The two complex varieties $M_B(C,d)$ and $M_B(C,d')$ can be obtained by base change of a scheme defined over $\BQ[\zeta]$ via two complex embeddings $\BQ[\zeta] \hookrightarrow \BC$; the two embeddings  differ by an automorphism of $\BQ[\zeta]$ sending $\zeta^d$ to $\zeta^{d'}$. Hence we have  identified the cohomology of the two moduli spaces as $\BC$-algebras induced by this automorphism of $\BQ[\zeta]$:
\begin{equation}\label{Gal0}
H^*(M_\mathrm{B}(C,d), \BC) \xrightarrow{\simeq} H^*(M_\mathrm{B}(C,d'), \BC).
\end{equation}
Furthermore, as an immediate consequence of the description of $\CT'_d$, the pairs $(M_{\mathrm{B}}(C,d), \CT'_d)$ and $(M_{\mathrm{B}}(C,d') ,\CT'_{d'})$  correspond via the automorphism $\zeta^d \mapsto \zeta^{d'}$ above. In particular, the isomorphism (\ref{Gal0}) preserves each class (\ref{taut_T}) if we pass through the non-abelian Hodge correspondence. Switching back to the tautological classes $c(\gamma,k)$, the \emph{only if} direction is now clear. The \emph{if} direction follows from Markman's theorem \cite{Markman} that the classes  $c(\gamma,k)$ generate the total cohomology as a $\BC$-algebra.  
\end{proof}

\subsection{Change the degree by $n$}
Let $\CO_C(1)$ be a degree 1 line bundle on the curve $C$. Taking tensor product with $\CO_C(1)$ induces an isomorphism between the moduli spaces
\begin{equation}\label{eq14}
M_{\mathrm{Dol}}(C,d) \xrightarrow{\simeq} M_{\mathrm{Dol}}(C,d+n), \quad (\CE, \theta) \mapsto (\CE \otimes \CO_C(1),\theta).
\end{equation}

\begin{prop}\label{prop1.3}
The isomorphism of the cohomology 
\[
H^*(M_{\mathrm{Dol}}(C,d), \BC) \xrightarrow{\simeq} H^*(M_{\mathrm{Dol}}(C,d+n), \BC)
\]
induced by (\ref{eq14}) preserves the tautological classes
\[
c(\gamma, k)  \in H^*(M_{\mathrm{Dol}}(C,d), \BC) \mapsto c(\gamma, k)  \in H^*(M_{\mathrm{Dol}}(C,d+n), \BC).
\]
\end{prop} 

\begin{proof}
Under the isomorphism (\ref{eq14}), a universal family for $M_{\mathrm{Dol}}(C,d+n)$ is obtained by taking the tensor product of a universal family for $M_{\mathrm{Dol}}(C,d)$ with the pullback of $\CO_C(1)$ from $C$. Hence the isomorphism 
\[
H^*(C \times M_{\mathrm{Dol}}(C,d), \BC) \xrightarrow{\simeq} H^*(C \times M_{\mathrm{Dol}}(C,d+n), \BC)
\]
induced by (\ref{eq14}) also preserves the normalized classes, and thus the tautological classes. \end{proof}

\section{Similitude groups}\label{Sec2}

\subsection{Good elements}\label{Good}
We fix two integers $d, d'$ coprime to $n$. We call an element $A \in \mathrm{GSp}(\Lambda_\BC)$ \emph{good} if there is an isomorphism of $\BC$-algebras
\[
G_A: H^*(M_{\mathrm{Dol}}(C,d), \BC) \xrightarrow{\simeq} H^*(M_{\mathrm{Dol}}(C,d'), \BC) 
\]
satisfying that:
\begin{enumerate}
    \item[(i)] $G_A(c(\gamma, k)) = c(A\gamma, k)$ for any $\gamma \in H^*(C, \BC)$ and $k\in \BN$; and
    \item[(ii)] $G_A\left( P_iH^j(M_{\mathrm{Dol}}(C,d), \BC) \right) = P_iH^j(M_{\mathrm{Dol}}(C,d'), \BC)$ for any $i,j \in \BN$.
\end{enumerate}

We denote by $\CG$ the set of all good elements. It is clear that Theorem \ref{main_thm} is equivalent to
\begin{equation}\label{thm2.1}
\CG = \mathrm{GSp}(\Lambda_\BC).
\end{equation}
We prove some basic properties for $\CG$ in Section \ref{Sec2.3}, and conclude the section with a criterion for (\ref{thm2.1}).


\subsection{Basic properties}\label{Sec2.3}

\begin{lem} \label{lemma1}
The set $\CG \subset \mathrm{GSp}(\Lambda_\BC)$ is closed under the left- or right-action of $\mathrm{Sp}(\Lambda)$.
\end{lem}

\begin{proof}
This is given by the monodromy symmetry. More precisely, for fixed $d$, a monodromy operator obtained by varying the curve $C$ is of the form
\[
H^*(M_{\mathrm{Dol}}(C,d), \BC) \xrightarrow{\simeq} H^*(M_{\mathrm{Dol}}(C,d), \BC), \quad c(\gamma, k) \mapsto c(M\gamma, k)
\]
with $M \in \mathrm{Sp}(V)$ (cf. the end of \S\ref{s0.4}), and it follows from \cite[Theorem 1.1.1]{dCM} that a monodromy operator preserves the perverse filtration. The lemma follows from composing $G_A$ with a monodromy operator on the left or right.
\end{proof}

\begin{lem}\label{lemma2}
The set $\CG \subset \mathrm{GSp}(\Lambda_\BC)$ is Zariski closed.
\end{lem}

\begin{proof}
For convenience, we pick a $\BC$-basis $e_1, e_2, \dots, e_{2g}$ of $H^1(C, \BC)$. So we have the $\BC$-algebra generators 
\begin{equation}\label{eqn14}
c(1, k),~~ c(e_i, k),~~ c(\mathrm{pt}, k),~~  \quad i,k \in \BN
\end{equation}
of $H^*(M_{\mathrm{Dol}}(C,d), \BC)$ and $H^*(M_{\mathrm{Dol}}(C,d'), \BC)$. In order to prove that $\CG$ is Zariski closed in $\mathrm{GSp}(\Lambda_\BC)$, it suffices to show that (i) and (ii) are Zariski closed conditions for $A \in \mathrm{GSp}(\Lambda_\BC)$.

We first treat (i). Any $c(A\gamma, k)$ can be expressed in terms of the classes (\ref{eqn14}) with coefficients given by certain entries of the matrix $A$. Hence the condition that $G_A$ is a $\BC$-algebra isomorphism sending $c(\gamma,k)$ to $c(A\gamma,k)$ is equivalent to that any relation between $\{c(\gamma, k)\}$ on $M_{\mathrm{Dol}}(C,d)$ are sent to a relation between $\{c(A\gamma, k)\}$ on $M_{\mathrm{Dol}}(C,d')$, which is clearly a Zariski closed condition on the entries of $A$.

For (ii), we consider a filtered basis of each cohomology group with respect to the perverse filtration; that is, we require the basis to satisfy that each piece $P_j$ of the perverse filtration is spanned by a subset of the basis. The condition (ii) can be expressed completely in terms of the filtered basis, namely, the image of every vector of the basis which lies in $P_j$ is sent to a vector which is only a linear combination of the sub-basis giving a basis of $P_j$. In other words, the condition (ii) is equivalent to the vanishing of certain coefficients for the linear transformation $G_A$ under the filtered basis. By Markman \cite{Markman} every vector of the filtered basis is expressed in terms of a linear combination of products of the classes (\ref{eqn14}). Hence coefficients for the linear transformation $G_A$ under the filtered basis are polynomials in entries of the matrix $A$, whose vanishing is also a Zariski closed condition.
\end{proof}

Combining the two lemmas above, we get the following criterion for $\CG = \mathrm{GSp}(\Lambda_\BC)$.

\begin{prop}\label{prop2.4}
If the set
\[
\big{\{}\lambda_A| ~~ A \in \CG \subset \mathrm{GSp}(\Lambda_\BC)\big{\}} \subset \BG_m
\]
is infinite, then $\CG = \mathrm{GSp}(\Lambda_\BC)$.
\end{prop}

\begin{proof}
The Borel density theorem entails that $\mathrm{Sp}(\Lambda) \subset \mathrm{Sp}(\Lambda_\BC)$ is a Zariski dense subset. We then deduce from Lemmas \ref{lemma1} and \ref{lemma2} that $\CG$ is preserved by the right action of $\mathrm{Sp}(\Lambda_\BC)$ on $\mathrm{GSp}(\Lambda_\BC)$. Hence to prove $\CG = \mathrm{GSp}(\Lambda_\BC)$, it suffices to prove that the image of $\CG$ via the projection
\[
\mathrm{GSp}(\Lambda_\BC) \to \mathrm{GSp}(\Lambda_\BC)/\mathrm{Sp}(\Lambda_\BC)
\]
is dense, where the quotient is with respect to the right action. This is exactly the assumption of the proposition, since the similitude character $\lambda_A$ defines an isomorphism 
\[
\mathrm{GSp}(\Lambda_\BC)/\mathrm{Sp}(\Lambda_\BC) = \BG_m.   \qedhere
\]
\end{proof}

Next, we use techniques from positive characteristic to construct sufficiently many elements in $\CG$.

\section{Reduction to positive characteristic}\label{Section3}


In this section, we will always use $\overline{\BQ}_{\ell}$-adic cohomology with  the prime $\ell$ coprime to $p$. In order to compare $\overline{\BQ}_{\ell}$-adic cohomology with singular cohomology with $\BC$-coefficients, we fix an isomorphism $\overline{\mathbb{Q}}_l \xrightarrow{\simeq} \BC$.


\subsection{Non-abelian Hodge in positive characteristic}
Our main tool to construct elements in $\CG$ is to use the non-abelian Hodge correspondence in positive characteristic \cite{Groch, Chen-Zhu, dCZ}.

Let $C_p$ be a curve over an algebraically closed field $\mathbf{k}$ of positive characteristic $p >0$. 
Throughout the rest of the paper, we assume that $p$ 
is large enough so that it is 
is coprime to the rank $n$. We denote by $C^{(1)}_p$ the Frobenius twist of the curve $C_p$ obtained from the base change of 
\[
\mathrm{Frob}: \mathbf{k} \rightarrow \mathbf{k}, \quad \quad x \mapsto x^p,
\]
and we denote by 
\[
\mathrm{Fr}_p: C_p \to C_p^{(1)}
\]
the relative Frobenius $\mathbf{k}$-morphism, which is finite of degree $p$. 


Identical to the case of $\BC$, the moduli space $M_{\mathrm{Dol}}(C^{(1)}_p, d)$ of (slope)-stable Higgs bundle of rank $n$ and degree $d$ over the curve $C^{(1)}_p$ carries a Hitchin fibration $h_p:M_{\mathrm{Dol}}(C^{(1)}_p, d) \to A(C_p^{(1)})$. Compared with the characteristic 0 case, a new feature in characteristic $p$ is the existence of a Hitchin type fibration, which is called the \emph{Hitchin-de Rham} morphism, from the moduli space $M_{\mathrm{dR}}(C_p, dp)$ of (slope-)stable flat connections
\begin{equation}\label{hitdr}
\nabla: \CE \to \CE \otimes \Omega_{C_p}, \quad \mathrm{rank}(\CE) = n,~~~ \mathrm{deg}(\CE) = dp
\end{equation}
to the Hitchin base $A(C_p^{(1)})$ associated with the Frobenius twist $C_p^{(1)}$. More concretely, the \emph{$p$-curvature} of a flat bundle gives rise to a morphism $\CE \to \CE \otimes \mathrm{Fr}_p^* \Omega_{C^{(1)}_p}$, whose characteristic polynomial induces the Hitchin-de Rham morphism
\[
h^{(1)}_p : M_{\mathrm{dR}}(C_p,dp) \rightarrow A(C_p^{(1)});
\]
we refer to \cite{LP} and \cite[Section 3]{Groch} for more details for the $p$-curvature and the Hitchin-de Rham fibration. Consequently, the cohomology of both moduli spaces $M_{\mathrm{Dol}}(C_p^{(1)}, d)$ and $M_{\mathrm{dR}}(C_p, dp)$ admit perverse filtrations induced by $h_p$ and $h_p^{(1)},$ respectively.

Groechenig's version of the  non-abelian Hodge theorem in characteristic $p$ \cite[Theorem 1.1]{Groch} asserts that the two morphisms  (Hitchin for $C_p^{(1)},$ and Hitchin-de Rham for $C_p$)
\begin{equation}\label{etale}
h_p:  M_{\mathrm{Dol}}(C^{(1)}_p, d) \to A(C_p^{(1)}), \quad h^{(1)}_p : M_{\mathrm{dR}}(C_p, dp) \rightarrow A(C_p^{(1)})
\end{equation}
are both proper and surjective, and are \'etale-locally equivalent over the base $A(C_p^{(1)})$.
We remark that Chen-Zhu \cite{Chen-Zhu}  proved a similar result for stacks (without semistability assertions)
for arbitrary reductive groups.

The Hitchin-de Rham fibration  (\ref{etale})  for $C_p$ in degree $dp$
\begin{equation}\label{3.2(1)}
h^{(1)}_p : M_{\mathrm{dR}}(C_p,dp) \rightarrow A(C_p^{(1)})
\end{equation}
is closely related to the Hitchin fibration for the Dolbeault moduli space for $C_p$ in  the same degree $dp$:
\begin{equation}\label{3.2(2)}
h_p: M_{\mathrm{Dol}}(C_p,dp) \rightarrow A(C_p).
\end{equation}
via the Hodge moduli space
\[
\tau: M_{\mathrm{Hod}}(C_p,dp) \rightarrow A(C_p^{(1)})\times \BA_t^1 \rightarrow \BA_t^1
\]
parameterizing $t$-connections; this was constructed in \cite{LP} in the degree 0 case and was extended to the case of $(n, dp)=1$ by \cite[Proposition 3.1]{dCZ}. More precisely, it was shown in \cite{dCZ} that $\tau$ is a smooth family over $\BA_t^1$; the fiber of $\tau$ over $t= 0 \in \BA_t^1$  recovers (\ref{3.2(1)});
 the fiber over $t= 1$ recovers (\ref{3.2(2)}), post-composed with the natural universal homeomorphism between the Hitchin bases $A(C^{(1)}_p) \xrightarrow{\simeq} A(C_p)$; this latter can be identified with the relative Frobenius for the Hitchin base $A(C_p)$ \cite[Lemma 4.3]{dCZ0}. 
 

\subsection{Cohomological correspondences}

Note that the correspondence (\ref{etale}) is local in nature.
Under the coprimality assumption  $(n,dp)=1$ for rank and degree,  de Cataldo and Zhang \cite{dCZ} established a series  of \emph{global} cohomological correspondences that we now describe.

The global nilpotent cones associated with the Dolbeault and the de Rham moduli spaces
\begin{equation}\label{nilp}
N_{\mathrm{Dol}}(C_p^{(1)},d): = h_p^{-1}(0), \quad \quad N_{\mathrm{dR}}(C_p,dp): = {h^{(1)}_p}^{-1}(0)
\end{equation}
are  isomorphic by the \'etale equivalence of (\ref{etale}). In fact, there is a distinguished isomorphism 
given by (\ref{disti}) that we will describe later. The cohomology rings of the global nilpotent cones (\ref{nilp}) carry  natural filtrations induced by the respective perverse filtrations
and the decomposition theorem; see \cite[Remark 2.3]{dCZ}.

We have the following  canonical commutative diagram of canonical ring isomorphisms (coefficients $\overline{\BQ}_{\ell}$ throughout),
which are also filtered isomorphisms for the respective perverse filtrations
\begin{equation}\label{cocorr}
\begin{tikzcd}
  H^* (M_{\mathrm Dol} (C_p, dp)) 
 &  H^* (M_{\mathrm dR} (C_p, dp))    \arrow[l,  "\mathrm{(d)}", "\simeq"'] 
    & H^*(N_{\mathrm dR} (C_p, dp))   \arrow[l,  "(\mathrm{c)}","\simeq"'] 
    \\
  H^*(M_{\rm Dol} (C_p^{(1)}, d))  \arrow[rr, "\simeq",  "\mathrm{(a)}"',] \arrow[u,  dotted, "\simeq",  "{\Theta}_M"',] &  & H^*(N_{\mathrm Dol} (C_p^{(1)}, d))  \arrow[u,  "\mathrm{(b)}","\simeq"'],
\end{tikzcd}
\end{equation}
where: the perverse filtrations on the cohomology groups  of
$N_{\mathrm{Dol}}$ and of $N_{\mathrm{dR}}$ are defined in \cite[\S2.2]{dCZ}; 
(a) is the filtered isomorphism in \cite[(33) 
combined with Remark 2.3]{dCZ}, and it is induced by restriciton;
(b) is the second filtered isomorphism in \cite[(36)]{dCZ} and it 
is  induced (this also uses results of M. Groechenig, as indicated in \textit{loc.\;cit.}) by the local non-abelian Hodge correspondence over $0 \in A(C_p)$ (see (\ref{disti})); 
(c) is the filtered isomorphism in  \cite[(31)]{dCZ}, and it is induced by restriciton;
(d) is the filtered isomorphism in  \cite[(25)]{dCZ}. See also the companion diagram (\ref{cocorr-b}).
The composition is the filtered ring isomorphism
\begin{equation}\label{eqn20}
{\Theta}_M:  H^*(M_{\mathrm{Dol}}(C_p^{(1)},d), \overline{\BQ}_l) \xrightarrow{\simeq} H^*(M_{\mathrm{Dol}}(C_p, dp), \overline{\BQ}_l).
\end{equation}

\subsection{Lift to characteristic zero via finite fields}\label{Sec3.3}

Given a curve $C$ over the complex numbers, for each prime $p$ with $p>n$ and $p \neq \ell$,
we can use the canonical isomorphism ${\Theta}_M$ to produce a \emph{non-canonical} ring isomorphism 
\begin{equation}\label{eqn21}
\begin{tikzcd}
 \widetilde{\Theta}_M:  H^* (M_{\mathrm Dol} (C, d), \BC)  \ar[r, "\simeq"]   &
  H^*(M_{\rm Dol} (C, dp),\BC), 
\end{tikzcd}
\end{equation}
filtered with respect to the respective perverse filtrations.

The construction proceeds as follows.  First, notice that, for a given $p$, it suffices to construct $\widetilde{\Theta}_M$
for a specific curve $C$ and then apply a monodromy argument, as in \cite{dCM}, or \cite[Proposition 3.3.(2)]{dCZ}, to extend to arbitrary curves.  Note that this introduces an ambiguity governed by the monodromy action
of $\mathrm{Sp}(V)$

We will choose $C$ to be the lift of a curve defined over $\BF_p$.
Let $C_p$ be a smooth curve defined over the finite field $\BF_p$; here for notational convenience, we denote the base change of  $C_p$ to an algebraic closure
$\mathbf{k}$ 
of $\BF_p$ by the same symbol $C_p.$  
By our choice, we have that $C_p = C_p^{(1)}.$
We can lift $C_p$ to a morphism $\CC \to \mathrm{Spec}(R)$ with $R$ a complete strictly Henselian DVR of characteristic 0 with residue field $\mathbf{k}$, where the geometric generic fiber $C$ is a connected nonsingular curve of genus $g\geq 2$.

 Specialization morphisms associated with the smooth family $\CC \to \mathrm{Spec}(R)$ induce canonical isomorphisms 
\begin{equation}\label{eqnn23}
H^*(C, \overline{\BQ}_{\ell}) = H^*(C_p, \overline{\BQ}_{\ell}), \quad H^*(M_{\mathrm{Dol}}(C,d), \overline{\BQ}_{\ell}) = H^*(M_{\mathrm{Dol}}(C_p, d), \overline{\BQ}_{\ell}),
\end{equation}
where we apply \cite[Proposition 3.3.(2)]{dCZ} for the second isomorphism. If we combine these with the canonical isomorphism
$\Theta_M$ (\ref{cocorr})
and tensoring with
the fixed isomorphism $\overline{\BQ}_{\ell} \xrightarrow{\simeq} \BC$,
we obtain the desired $\widetilde{\Theta}_M$.

\begin{thm}\label{thm3.1}
The ring isomorphism $\widetilde{\Theta}_M$ satisfies
\[
\widetilde{\Theta}_M \big{(}c(\gamma, k)\big{)} = c\big{(}p^{-1} \mathrm{Fr}_p^*\gamma,k\big{)}, \quad \quad \forall \gamma \in H^*(C, \BC) ,~~~\forall k \in \BN.
\]
Here, $\mathrm{Fr}_p^*: H^*(C_p, \BC) \stackrel{\sim}\to H^*(C_p, \BC)$ is the Frobenius pullback associated with the relative Frobenius $\mathbf{k}$-morphism $\mathrm{Fr}_p: C_p \to C_p^{(1)}=C_p,$ and we have used the fixed isomorphism $\overline{\BQ}_{\ell} \xrightarrow{\simeq} \BC$.
\end{thm}

\begin{rmk}\label{rmk3.1}
If we view the operator $A_p:= p^{-1}\mathrm{Fr}^*_p: H^*(C, \BC) \to H^*(C, \BC)$ as an element in $\mathrm{GSp}(\Lambda_\BC)$, then by the peculiar action defined in \S\ref{s0.4}, it acts as $\mathrm{id}$ on $H^2(C,\BC)$ and as multiplication by $\lambda_{A_p} = p^{-1}$ on $H^0(C,\BC),$ thus justifying introducing this action.\end{rmk}



\section{Proof of Theorem \ref{thm3.1}} \label{Sec4}


\subsection{Isomorphisms}
In the setting of Section \ref{Sec3.3}, we consider the relative  Dolbeault moduli space for degree $d$
\begin{equation}\label{eqnn25}
M_{\mathrm{Dol}}(\CC/R, d) \to \mathrm{Spec}(R)
\end{equation}
that is smooth over $R$, and the Hodge moduli space that is smooth over $\BA_t^1$
(\cite[Proposition 3.1]{dCZ}):
\begin{equation}\label{eqnn26}
    \tau: M_{\mathrm{Hod}}(C,dp) \rightarrow \BA_t^1.
\end{equation}
The morphism $\widetilde{\Theta}_M$ for Theorem \ref{thm3.1} is obtained as a composition, by inserting
the Dolbeault moduli spaces for $C$ into (\ref{cocorr}), by using (\ref{eqnn23}), and by keeping in mind that $C_p=C_p^{(1)}$, 
\begin{equation}\label{cocorr-b}
\begin{tikzcd}
  H^* (M_{\mathrm{Dol}} (C, dp))  & H^* (M_{\mathrm{Dol}} (C_p, dp))   \arrow[l, "\simeq"', "(\ref{eqnn23})~ \mathrm{u_2}"] 
 &  H^* (M_{\mathrm{dR}} (C_p, dp))    \arrow[l,  "\mathrm{(d)}", "\simeq"'] 
    & H^*(N_{\mathrm{d}R} (C_p, dp))   \arrow[l,  "\mathrm{(c)}","\simeq"'] 
    \\
  H^* (M_{\mathrm{Dol}} (C, d)) \arrow[r, "\simeq", "(\ref{eqnn23})~ \mathrm{u_1}"']  \arrow[u,  dotted, "\simeq",  "\widetilde{{\Theta}}_M"',] & H^*(M_{\rm Dol} (C_p,d))  \arrow[rr, "\simeq",  "\mathrm{(a)}"',] \arrow[u, "\simeq",  "{\Theta}_M"',] &  & H^*(N_{\mathrm{Dol}} (C_p, d))  \arrow[u,  "\mathrm{(b)}","\simeq"'].
\end{tikzcd}
\end{equation}
Here: 
$\mathrm{u_1}$ and $\mathrm{u_2}$  are the 
specialization maps (\ref{eqnn23}) associated with (\ref{eqnn25}), for degrees $d$ and $dp$ respectively; $\mathrm{(a)}$ and $\mathrm{(c)}$ are the natural restriction maps for the global nilpotent cones
(\ref{nilp}); $\mathrm{(d)}$  is the specialization map associated with (\ref{eqnn26}); $\mathrm{(b)}$ is 
 induced by the local non-abelian Hodge correspondence over $0 \in A(C_p)$ (see (\ref{disti})); $\widetilde{\Theta}_M$
 is the composition.

For our purpose, we calculate each map with respect to the tautological classes. 

\begin{rmk}\label{rmk4.1}
We note that by \cite[Lemma 3.1 and the proof of Corollary 3.2]{Heinloth}, the fact that the rank and the degree are coprime ensures that each of the Dolbeault, de Rham, and Hodge moduli spaces above carries a \emph{universal family}. In particular, we may define the tautologial classes $c(\gamma,k)$ for each cohomology in the chain above via a universal family, compatibly with restriction maps.
\end{rmk}

\subsection{Restrictions and specializations} We first note that restrictions $\mathrm{(c,a)}$ and specialization maps 
$\mathrm{(u_1,d,u_2)}$ preserve the tautological classes, \emph{i.e.}, they send $c(\gamma,k)$ to $c(\gamma, k)$.

This statement is clear for restriction maps, since the restriction of a universal family on $M_{\mathrm{Dol}}(C_p,d)$ (resp. $M_{\mathrm{dR}}(C_p,d)$) to the corresponding global nilpotent cone $N_{\mathrm{Dol}}(C_p,d)$ (resp. $N_{\mathrm{dR}}(C_p,d)$) is still a universal family.

For the specialization maps $\mathrm{u_1}, \mathrm{u_2}$ and $\mathrm{(d)}$, this follows from the existence of universal families over $\mathrm{Spec}(R)$ and $\BA^1_t$ respectively; see Remark \ref{rmk4.1}.


\subsection{Global nilpotent cones}

Finally, we treat the morphism $\mathrm{(b)}$ in (\ref{cocorr-b}). 
Our goal is to prove the identities (\ref{zz2}) and (\ref{zz4}), to be used  in \S\ref{pf3.1} when
 proving  Theorem \ref{thm3.1}.
 
In order to carry out the computation, we need the precise description of the non-abelian Hodge correspondence \cite[Corollary 3.28, Lemma 3.46]{Groch} for the global nilpotent cones, which we review briefly as follows.

Let $D_{C_p}$ be the sheaf of crystalline differential operators on the curve $C_p= C_p^{(1)}.$
The pushforward of $D_{C_p}$ along the Frobenius map $\mathrm{Fr}_p: C_p \to C_p$ satisfies 
\[
\mathrm{Fr}_{p*}D_{C_p} = \pi_* \CD,
\]
where $\pi$ is the projection $T^*C_p \to C_p$ and $\CD$ is a uniquely determined $\CO_{T^*C_p}$-algebra; see \cite[Lemma 2.8]{Groch}. Moreover, by \cite[Theorem 3.20]{Groch}, the restriction of $\CD$ to any spectral curve $C_\alpha \subset T^*C_p$  splits, \emph{i.e.} we have an isomorphism 
\begin{equation}\label{splitting0}
\CD|_{C_\alpha} = \mathcal{E}nd_{\CO_{C_\alpha}}(V_\alpha)
\end{equation}
with $V_\alpha$ a rank $p$ vector bundle  on $C_\alpha$, well-defined up to tensoring with a line bundle.

We denote by $C_n \subset T^*C_p$ the spectral curve associated with $0 \in A(C_p)$;  it is the $n$-th thickening of the 0-section $C_p \subset T^*C_p$. By \cite[Proof of Corollary 3.45]{Groch}, there is a canonical choice of a vector bundle $V_n$ of rank $p$ that splits $\CD|_{C_n}$ as in (\ref{splitting0}). Such a choice induces a canonical isomorphism $\nu$ between the global nilpotent cones $N_{\mathrm{Dol}}(C_p, d)$ and $N_{\mathrm{dR}}(C_p, dp)$. For our purposes, we need to describe  the interaction of universal families under this isomorphism.

By the classical BNR correspondence, a Higgs bundle in $N_{\mathrm{Dol}}(C_p, d)$ is given as $(\pi_*\CF, \theta)$ where $\CF$ is a stable pure 1-dimensional sheaf supported on $C_n,$ and  $\pi: C_n \to C_p$ is the projection. It corresponds to a flat connection $(\CE, \nabla)\in N_{\mathrm{dR}}(C_p, dp)$ satisfying 
\begin{equation}\label{compare1}
\mathrm{Fr}_{p*}\CE \simeq \pi_* (\CF \otimes V_n).
\end{equation}

To globalize (\ref{compare1}) over the moduli spaces,
we denote by $\nu$ the non-abelian Hodge correspondence for the global nilpotent cones induced by $V_n$:
\begin{equation}\label{disti}
\nu: N_{\mathrm{dR}}(C_p, dp) \xrightarrow{\simeq} N_{\mathrm{Dol}}(C_p, d), \quad (\CE, \nabla) \mapsto (\pi_*\CF, \theta)
\end{equation}
where $\CE$ and $\CF$ satisfy (\ref{compare1}).
The isomorphism (b) in (\ref{cocorr-b}) is $\nu^*$ with inverse $\nu_*$.
We consider the degree $p$ finite map
\begin{equation}\label{eqnF}
F:= \mathrm{Fr}_p \times \nu: C_p \times  N_{\mathrm{dR}}(C_p, dp) \to C_p \times N_{\mathrm{Dol}}(C_p, d),
\end{equation}
and the projection 
\[
\widetilde{\pi}: C_n \times  N_{\mathrm{Dol}}(C_p, d) \to C_p \times N_{\mathrm{Dol}}(C_p, d).
\]
We also have the following natural projections:
\[
\begin{split}
p_{\mathrm{dR}}: C_p \times  N_{\mathrm{dR}}(C_p, dp) \to C_p,& \quad q_{\mathrm{dR}}: C_p \times  N_{\mathrm{dR}}(C_p, dp) \to N_{\mathrm{dR}}(C_p, dp),\\
p_{\mathrm{Dol}}: C_p \times  N_{\mathrm{Dol}}(C_p, d) \to C_p,& \quad q_{\mathrm{Dol}}: C_p \times  N_{\mathrm{Dol}}(C_p, d) \to N_{\mathrm{Dol}}(C_p, d),\\
r: C_n \times  N_{\mathrm{Dol}}(C_p, dp) \to C_n.
\end{split}
\]

 Let $\CV$ be the pullback $r^* V_n$ of $V_n$ along the projection $r$; it is a vector bundle of rank $p$. Let $(\CU_{\mathrm{dR}} ,\nabla)$ and
\begin{equation}
    \label{dolpifn}
    (\CU_{\mathrm{Dol}}, \theta) = (\widetilde{\pi}_* \CF_n, \theta)
\end{equation}
be universal families over the schemes on both sides of (\ref{eqnF}), where $\CF_n$ is a universal family over $C_n \times N_{\mathrm{Dol}}(C_p, d)$. The isomorphism (\ref{compare1}) shows that the pushforward $F_*\CU_{\mathrm{dR}}$ and $\widetilde{\pi}_*(\CF_n \otimes \CV)$ coincide after restricting over any closed point on $N_{\mathrm{Dol}}(C_p, d)$. Therefore $F_*\CU_{\mathrm{dR}}$ is the tensor product of $\widetilde{\pi}_*(\CF_n \otimes \CV)$ and a line bundle pulled back from $N_{\mathrm{Dol}}(C_p,d)$. Modifying the universal family $\CF_n$ by this line bundle, we may assume that the universal families $\CU_{\mathrm{dR}}$ and $\CF_n$ satisfy 
\begin{equation}\label{key_equation}
    F_* \CU_{\mathrm{dR}} \simeq \widetilde{\pi}_*(\CF_n \otimes \CV).
\end{equation}

Recall the following type of classes (cf. (\ref{form0})), which play a role in producing normalized and tautological classes.
Let $\alpha_0 \in H^2(C \times N_{\mathrm{Dol}}(C_p,d), \overline{\BQ}_{\ell})$ be a class of the form
\begin{equation}\label{form}
\alpha_0 =  p_{\mathrm{Dol}}^* \alpha_0'+ q_{\mathrm{Dol}}^* \alpha_0''.
\end{equation}

We need the following two classes of the form (\ref{form}) (for the first, Dol is replaced by dR) 
\begin{equation}\label{right form}
 \alpha_1:=\frac{p-1}{2}p_{\mathrm{dR}}^*c_1(\omega_{C_p}), \quad
 \beta:=\frac{p-1}{2p}p_{\mathrm{Dol}}^*c_1(\omega_{C_p}).
 \end{equation}

By explicit calculation, we have the identity   $\mathrm{td}(T_F)= 1+\alpha_1$ for 
the Todd class $\mathrm{td}(T_F)$ of the virtual tangent bundle $T_F$ (cf. \cite[B.7.6]{Fulton}).
The class $\beta$ will appear naturally shortly.

\begin{lem}\label{lemma z} We have the following identities
\begin{equation}\label{zz1}
F_*[ \mathrm{ch}^{\alpha_1} (\CU_{\mathrm{dR}})]
=
\mathrm{ch}  (F_* \CU_{\mathrm{dR}}) = 
p\cdot \mathrm{ch}^\beta (\CU_{\mathrm{Dol}}).
\end{equation}
\end{lem}
\begin{proof}
The first identity follows
by applying Grothendieck--Riemann--Roch  (GRR) (cf. \cite[Ex. 18.3.10]{Fulton}) to  the l.c.i. morphism $F$ (\ref{eqnF})
    \begin{equation}
    \label{chfudr}
         \mathrm{ch}(F_* \CU_{\mathrm{dR}}) 
         =F_* \left[  \mathrm{ch}(\CU_{\mathrm{dR}}) \mathrm{td}(T_F) \right]
         = F_* \left[  \mathrm{ch}^{\alpha_1}(\CU_{\mathrm{dR}})\right].
    \end{equation}

To argue for the second identity, we need three useful facts: 
Firstly,
since the inclusion $i: C_p\hookrightarrow C_n$ is a section to the projection $\pi: C_n\to C_p$, we have that 
the isomorphism $i^*: H^2(C_n)\to H^2(C_p)$ is the inverse to
the isomorphism $\pi^*: H^2(C_p)\to H^2(C_n)$.
Secondly,
the splitting $V_n$ on $C_n$ that is chosen in \cite[Proof of Corollary 3.45]{Groch} satisfies the identity of vector bundles $i^*V_n=\mathrm{Fr}_{p*}\mathcal{O}_{C_p}$. 
Thirdly, we have $c_1(\mathrm{Fr}_{p*}\mathcal{O}_{C_p})=\frac{p-1}{2}c_1(\omega_{C_p})$, thus we have that $p\beta=p_{\mathrm{Dol}}^*c_1(\mathrm{Fr}_{p*}\mathcal{O}_{C_p})$.
With these three facts, we obtain the following equalities:
\begin{equation}
\begin{split}
    \mathrm{ch}(F_*\mathcal{U}_{\mathrm{dR}}) & =\mathrm{ch}(\widetilde{\pi}_*(\mathcal{F}_n\otimes \mathcal{V}))=\widetilde{\pi}_*[\mathrm{ch}(\mathcal{F}_n)\mathrm{td}(T_{\widetilde{\pi}})\mathrm{ch}(\mathcal{V})]\\
    & =\widetilde{\pi}_*[\mathrm{ch}(\mathcal{F}_n)\mathrm{td}(T_{\widetilde{\pi}}) r^*\pi^*(p+i^*c_1(V_n))]\\
    &= \widetilde{\pi}_*[\mathrm{ch}(\mathcal{F}_n)\mathrm{td}(T_{\widetilde{\pi}}) \widetilde{\pi}^*p_{\mathrm{Dol}}^*(p+c_1(\mathrm{Fr}_{p*}\mathcal{O}_{C_p}))]\\
    &= \widetilde{\pi}_*[\mathrm{ch}(\mathcal{F}_n)\mathrm{td}(T_{\widetilde{\pi}})](p+p\beta)\\
    &= p\cdot \mathrm{ch}(\widetilde{\pi}_* \mathcal{F}_n)(1+\beta)= p\cdot \mathrm{ch}^{\beta}(\mathcal{U}_{\mathrm{Dol}}),
\end{split}
\end{equation}
where the first equality follows from (\ref{key_equation}); the second follows from GRR; the third follows from the definition $\mathcal{V}:=r^* V_n$ and the first useful fact above; the fourth follows from the identity $\pi\circ r=p_{\mathrm{Dol}}\circ \widetilde{\pi}: C_n\times N_{\mathrm{Dol}} \to C_p$ and second useful fact above; the fifth follows from the projection formula and the third useful fact; the sixth follows from GRR; and the last follows from  (\ref{dolpifn}). \qedhere
\end{proof}

By twisting (\ref{zz1}) with a class $\alpha_0$ of the form (\ref{form}), and by using the projection formula,  we get the identity:
 \begin{equation}\label{zz2}
F_*[ \mathrm{ch}^{F^* \alpha_0 +\alpha_1} (\CU_{\mathrm{dR}})]
=
\mathrm{ch}^{\alpha_0}  (F_* \CU_{\mathrm{dR}}) = 
p\cdot \mathrm{ch}^{\alpha_0+\beta} (\CU_{\mathrm{Dol}}).
\end{equation}

By integrating the lhs of (\ref{zz2})  over a class $\gamma \in H^*(C_p, \overline{\BQ}_{\ell}),$
the projection formula, together with the identities $q_{\mathrm{Dol}}\circ F = \nu \circ q_{\mathrm{dR}} $ and $\mathrm{Fr}_p \circ p_{\mathrm{dR}} = p_{\mathrm{Dol}}\circ F,$ yields the identities
    \[
    \begin{split}
         \int_{\gamma} \mathrm{ch}^{\alpha_0}(F_* \CU_{\mathrm{dR}})  & = q_{\mathrm{Dol}*}\left( p_{\mathrm{Dol}}^*\gamma  \cup   F_* \left[  \mathrm{ch}^{F^*\alpha_0+\alpha_1}(\CU_{\mathrm{dR}})\right] \right)\\
        & = q_{\mathrm{Dol}*} F_* \left( F^* p_{\mathrm{Dol}}^*\gamma  \cup   \mathrm{ch}^{F^*\alpha_0+\alpha_1}(\CU_{\mathrm{dR}}) \right) \\
        & =  \nu_* q_{\mathrm{dR}*} \left( p_{\mathrm{dR}}^* \mathrm{Fr}_p^*\gamma  \cup   \mathrm{ch}^{F^*\alpha_0+\alpha_1}(\CU_{\mathrm{dR}}) \right)\\
        &= \nu_* \int_{\mathrm{Fr}_p^*\gamma} \mathrm{ch}^{F^*\alpha_0+\alpha_1}(\CU_{\mathrm{dR}}).
    \end{split}
    \]
 
By  integrating the right-hand side of (\ref{zz2}), and by using
 that $\nu_*$ and $\nu^*$ are mutual inverses, we finally get the following identity for every $\alpha_0$ and $\gamma:$
  \begin{equation}\label{zz4}
 \int_{\mathrm{Fr}_p^* \gamma}  \mathrm{ch}^{F^* \alpha_0 +\alpha_1} (\CU_{\mathrm{dR}})
= 
\nu^*  [ p\cdot  \int_{\gamma} \mathrm{ch}^{\alpha_0+\beta} (\CU_{\mathrm{Dol}})].
\end{equation}

\subsection{Proof of Theorem \ref{thm3.1}}\label{pf3.1}
Recall that the isomorphism (b) in $(\ref{cocorr-b})$ is $\nu^*$ for which we have the identity (\ref{zz4}).

We claim that there is a necessarily unique class $\alpha_0$
of the form (\ref{form}) such that the two Chern characters in the first and third term of (\ref{zz2})  are \emph{simultaneously} normalized.

We pick $\alpha_0$ to be the unique class of the form (\ref{form}) so that $\mathrm{ch}^{F^*\alpha_0+ \alpha_1}_1(\mathcal{U}_{dR})$
lies in the K\"unneth component
\[
H^1(C_p)\otimes H^1(N_{\mathrm{dR}}(C_p,dp)) \subset H^2(C_p \times N_{\mathrm{dR}}(C_p,dp)),
\]
and consequently  $\mathrm{ch}^{F^*\alpha_0 +\alpha_1}(\mathcal{U}_{dR})$ is normalized.  To do so, we use that $F$ is a product and that  $F^*$ is an isomorphism.

Since $F=\mathrm{Fr}_p\times\nu,$ we have that $F_*$ respects the K\"unneth components. Therefore the class $\mathrm{ch}^{\beta+\alpha_0}(\mathcal{U}_{\mathrm{Dol}})$ is also normalized, and our claim is proved.

Since both Chern characters are normalized, it follows from (\ref{zz4}) that the isomorphism $\nu^*$ respects tautological classes. Moreover, in view of (\ref{zz4}), we obtain
 the following explicit expression for the isomorphism $\widetilde{\Theta}_M$ evaluated on the tautological classes
\[\widetilde{\Theta}_M(c(\gamma,k))=\Theta_M(c(\gamma,k))=\nu^* \int_{\gamma} \mathrm{ch}^{\beta+\alpha_0}(\CU_{\mathrm{Dol}})=  \int_{\frac{1}{p}\cdot\mathrm{Fr}_p^*\gamma} \mathrm{ch}^{F^*\alpha_0+\alpha_1}(\CU_{\mathrm{dR}})=c(p^{-1}\mathrm{Fr}_p^*\gamma,k).\]
Here we abuse the notation by omitting the horizontal isomorphisms in (\ref{cocorr-b}), which all respect 
normalized and tautological classes.
\qed

\section{Proofs of the main theorems and of Deligne splittings}

\subsection{Proof of Theorem \ref{main_thm}}\label{Sec5.1}
In view of Proposition \ref{prop2.4}, it suffices to show that there are infinitely many $A \in \CG$ with distinct $\lambda_A \in \BG_m$.

We first note that, if for one complex curve $C$ of genus $g\geq 2$ we find an isomorphism of the form $G_A$ as in Section \ref{Good}, then by composing with parallel transport operators we obtain an isomorphism of the same form for any nonsingular complex curve of genus $g$. Therefore, we may choose any curve to construct elements in $\CG$.

For any prime $p>n$ with
\begin{equation}\label{mod}
dp = d'~~~~ \textup{mod}~~~n,
\end{equation}
we work with $C$ and its reduction  $C_p$ as in Section \ref{Sec3.3}. By Proposition \ref{prop1.3} and Theorem \ref{thm3.1}, we obtain that
\[
A_p = p^{-1}  \mathrm{Fr}_p^* \in \CG ,\quad \textup{with}~~~ \lambda_{A_p} = p^{-1} \in \BG_m.
\]
The theorem follows since there are infinitely many primes $p$ satisfying (\ref{mod}). \qed

\subsection{Proofs of Theorems \ref{thm0.2} and \ref{thm0.4}}\label{5.2}

Theorem \ref{thm0.2} is deduced from Theorem \ref{main_thm}, since it follows from Proposition \ref{Prop1.2} that Galois conjugation (\ref{Galoi0}) coincides with the isomorphism ${\phi}_{d,d'}$ of Theorem \ref{main_thm} (b).

Now we prove Theorem \ref{thm0.4}. We consider the action of the 1-dimensional sub-group 
\[
\BT: = \{\lambda\mathrm{Id}_{2g}|~~\lambda \neq 0\} \subset \mathrm{GSp}(\Lambda_\BC)
\]
on the total cohomology $H^*(M_{\mathrm{Dol}}(C,d), \BC)$ induced by the action of Theorem \ref{main_thm} (a). We note that the weight decomposition of this $\BT$-action recovers the Hodge--Tate decomposition (\ref{Hodge-Tate}) via non-abelian Hodge (\ref{non-ab}). More precisely, we have
\[
\mathrm{Hdg}^i_j = \{\omega \in H^i(M_{\mathrm{Dol}}(C,d), \BC)|~~~ \lambda\cdot \omega = \lambda^{2j-i}\omega, ~~\forall \lambda \in \BT\};
\]
this follows directly from the explicit description of the $\BT$-action on the tautological classes:
\[
\lambda \cdot c(\gamma,k) = \lambda^{2-e}c(\gamma,k),\quad \quad \forall \gamma\in H^e(C,\BC), ~~~ \forall \lambda \in \BT.
\]
Furthermore, by Theorem \ref{main_thm} (a) each piece of the perverse filtration preserves the $\BT$-action. Hence the weight decomposition of $\BT$ on each piece $P_kH^i(M_{\mathrm{Dol}}(C,d), \BC)$ induces the desired decomposition
\[
P_kH^i(M_{\mathrm{Dol}}(C,d), \BC) = \bigoplus_{j}  \left(P_kH^i(M_{\mathrm{Dol}}(C,d), \BC) \cap \mathrm{Hdg}^i_j \right).
\]
This completes the proof of Theorem \ref{thm0.4}. \qed

\subsection{Extension to Deligne splitting}\label{5.3}

As we will recall briefly, the perverse filtration on the cohomology $H^*(M_{\mathrm{Dol}}(C,d),\BC)$ admits a natural splitting, known as the first Deligne splitting \cite{dC,D}.  Conjecturally this splitting corresponds, via Non-Abelian Hodge Theory, with the Hodge-Tate splitting of the weight filtration on $H^*(M_B(C,d),\BC)$ in (\ref{Hodge-Tate}).   

We first review some background.  Suppose we are given a triple $(H, F, \eta)$ where $H$ is a $\BC$-algebra, $F$ an increasing filtration on $H$ concentrated in degrees $[0,2r]$, and $\eta$ is an element of $H$.  We say that $\eta$ is an $F$-Lefschetz class if (i) multiplication by $\eta$ maps $F_k H$ to $F_{k+2}H$ and (ii) multiplication by $\eta^i$ induces isomorphisms
$\eta^i: Gr^{F}_{r-i}H \simeq Gr^{F}_{r+i}H$.
An $F$-Lefschetz class $\eta$ on $H$ induces, by means of an explicit linear algebra construction, a natural splitting of $F$ called the first Deligne splitting \cite{dC}.  In our setting, $H= H^*(M_{\mathrm{Dol}}(C,d),\BC)$ is the cohomology of the Dolbeault moduli space equipped with its perverse filtration $P$, and the $P$-Lefschetz class $\eta$ is a degree $2$ cohomology class.  The first Deligne splitting satisfies the following natural compatibility. 
Suppose we are given two such triples $(H, F, \eta)$ and $(H', F', \eta')$ and a ring isomorphism $f: H \xrightarrow{\simeq} H'$, which is a filtered isomorphism, for which $f(\eta) = \eta'$.  It is clear that if $\eta$ is an $F$-Lefschetz class then $\eta'$ is an $F'$-Lefschetz class and moreover that $f$ preserves the corresponding Deligne splittings.   

This compatibility gives us the following corollary of our main theorem.
\begin{cor}\label{firstdsp}
The operators $G_A$ of Theorem \ref{main_thm} 
preserve the first Deligne splittings for $H^*(M_{\mathrm{Dol}}(C,d),\BC)$.
\end{cor}
\begin{proof}
As explained in \cite[Remark 3.5]{dCMS}, in the case of $H^*(M_{\mathrm{Dol}}(C,d),\BC)$, the first Deligne splitting is \emph{independent} of the choice of $P$-Lefschetz class $\eta$.  In other words, we have naturally defined subspaces $S_j \subset H$ for which
$$P_k = \bigoplus_{j \leq k} S_j.$$
By the compatibility stated before the proposition, given two Dolbeault moduli spaces $M_{\mathrm{Dol}}(C,d)$ and 
$M_{\mathrm{Dol}}(C,d')$, 
any ring isomorphism 
$$
H^*(M_{\mathrm{Dol}}(C,d),\BC) \xrightarrow{\simeq} H^*(M_{\mathrm{Dol}}(C,d'),\BC),
$$
which is a filtered isomorphism, automatically preserves the first Deligne splitting and the graded pieces $S_j$.
In particular, this applies to the operators $G_A$ constructed in this paper.  
\end{proof}

\end{document}